\documentclass[a4paper,twoside]{article}
\usepackage{a4}
\usepackage{amssymb}
\usepackage{amsmath}
\usepackage[active]{srcltx}

\allowdisplaybreaks[2] 
\newcount\minutes \newcount\hours
\hours=\time
\divide\hours 60
\minutes=\hours
\multiply\minutes -60
\advance\minutes \time
\newcommand{\klockan}{\the\hours:{\ifnum\minutes<10 0\fi}\the\minutes}
\newcommand{\tid}{\today\ \klockan}
\newcommand{\prtid}{\smash{\raise 10mm \hbox{\LaTeX ed \tid}}}
\renewcommand{\prtid}{}
\makeatletter
\def\sectionmark#1{} 
\def\subsectionmark#1{}
\newcommand{\sectnr}{\ifnum \c@secnumdepth >\z@
                 \thesection.\hskip 1em\relax \fi}
\def\@evenhead{\footnotesize\rm\thepage\hfil\leftmark\hfil\llap{\prtid}}
\def\@oddhead{\footnotesize\rm\rlap{\prtid}\hfil\rightmark\hfil\thepage}
\def\tableofcontents{\section*{Contents} 
 \@starttoc{toc}}
\makeatother
\makeatletter
\def\@biblabel#1{#1.}
\makeatother
\makeatletter
\let\Thebibliography=\thebibliography
\renewcommand{\thebibliography}[1]{\def\@mkboth##1##2{}\Thebibliography{#1}
\addcontentsline{toc}{section}{References}
\frenchspacing 
\setlength{\@topsep}{0pt}
\setlength{\itemsep}{0pt}%
\setlength{\parskip}{0pt plus 2pt}%
}
\makeatother
\makeatletter
\def\mdots@{\mathinner.\nonscript\!.%
 \ifx\next,.\else\ifx\next;.\else\ifx\next..\else
 \nonscript\!\mathinner.\fi\fi\fi}
\let\ldots\mdots@
\let\cdots\mdots@
\let\dotso\mdots@
\let\dotsb\mdots@
\let\dotsm\mdots@
\let\dotsc\mdots@
\def\vdots{\vbox{\baselineskip2.8\p@ \lineskiplimit\z@
    \kern6\p@\hbox{.}\hbox{.}\hbox{.}\kern3\p@}}
\def\ddots{\mathinner{\mkern1mu\raise8.6\p@\vbox{\kern7\p@\hbox{.}}%
    \raise5.8\p@\hbox{.}\raise3\p@\hbox{.}\mkern1mu}}
\makeatother
\makeatletter
\let\Enumerate=\enumerate
\renewcommand{\enumerate}{\Enumerate%
\setlength{\@topsep}{0pt}
\setlength{\itemsep}{0pt}%
\setlength{\parskip}{0pt plus 1pt}%
\renewcommand{\theenumi}{\textup{(\alph{enumi})}}%
\renewcommand{\labelenumi}{\theenumi}%
}
\let\endEnumerate=\endenumerate
\renewcommand{\endenumerate}{\endEnumerate\unskip}
\makeatother
\makeatletter
\def\@seccntformat#1{\csname the#1\endcsname.\quad}
\makeatother
\makeatletter
\renewcommand*\l@section[2]{%
  \ifnum \c@tocdepth >\z@
    \addpenalty\@secpenalty
    \addvspace{0em \@plus\p@}
    \setlength\@tempdima{1.5em}%
    \begingroup
      \parindent \z@ \rightskip \@pnumwidth
      \parfillskip -\@pnumwidth
      \leavevmode \bfseries
      \advance\leftskip\@tempdima
      \hskip -\leftskip
      #1\nobreak\hfil \nobreak\hb@xt@\@pnumwidth{\hss #2}\par
    \endgroup
  \fi}
\makeatother
\newcommand{\authortitle}[2]{\author{#1}\title{#2}\markboth{#1}{#2}}
\newcommand{\auth}[2]{{#1, #2.}}
\newcommand{\idxauth}{\auth}
\newcommand{\art}[6]{{\sc #1, \rm #2, \it #3 \bf #4 \rm (#5), \mbox{#6}.}}
\newcommand{\artin}[3]{{\sc #1, \rm #2, in #3.}}
\newcommand{\artnopt}[6]{{\sc #1, \rm #2, \it #3 \bf #4 \rm (#5), \mbox{#6}}}

\newcommand{\book}[3]{{\sc #1, \it #2, \rm #3.}}

\newcommand{\AND}{{\rm and }}
\RequirePackage{amsthm}
\newtheoremstyle{descriptive}%
  {\topsep}   
  {\topsep}   
  {\rmfamily} 
  {}          
  {\bfseries} 
  {.}         
  { }         
  {}          
\newtheoremstyle{propositional}%
  {\topsep}   
  {\topsep}   
  {\itshape}  
  {}          
  {\bfseries} 
  {.}         
  { }         
  {}          
\newtheoremstyle{remarkstyle}%
  {\topsep}   
  {\topsep}   
  {\rmfamily}  
  {}          
  {\itshape} 
  {.}         
  { }         
  {}          
\theoremstyle{propositional}
\newtheorem{thm}{Theorem}[section]
\newtheorem{prop}[thm]{Proposition}
\newtheorem{lem}[thm]{Lemma}
\newtheorem{cor}[thm]{Corollary}
\theoremstyle{descriptive}
\newtheorem{deff}[thm]{Definition}

\makeatletter
\renewenvironment{proof}[1][\proofname]{\par
  \pushQED{\qed}%
  \normalfont 
  \trivlist
  \item[\hskip\labelsep
        \itshape
    #1\@addpunct{.}]\ignorespaces
}{%
  \popQED\endtrivlist\@endpefalse
}
\makeatother
\newdimen\extrawidth
\def\iintlim#1#2{\setbox0\hbox{$\scriptstyle#1$}%
        \setbox1\hbox{$\scriptstyle#2$}%
        \extrawidth=\wd1 \advance\extrawidth-\wd0
        \ifdim\extrawidth<0pt \extrawidth=0pt\fi%
        \int_{#1\kern\extrawidth \kern .5em}^{#2\kern -\wd1} \kern -.5em%
}
\newcommand{\setm}{\setminus}

\renewcommand{\phi}{\varphi}
\def\vint{\mathop{\mathchoice%
          {\setbox0\hbox{$\displaystyle\intop$}\kern 0.22\wd0%
           \vcenter{\hrule width 0.6\wd0}\kern -0.82\wd0}%
          {\setbox0\hbox{$\textstyle\intop$}\kern 0.2\wd0%
           \vcenter{\hrule width 0.6\wd0}\kern -0.8\wd0}%
          {\setbox0\hbox{$\scriptstyle\intop$}\kern 0.2\wd0%
           \vcenter{\hrule width 0.6\wd0}\kern -0.8\wd0}%
          {\setbox0\hbox{$\scriptscriptstyle\intop$}\kern 0.2\wd0%
           \vcenter{\hrule width 0.6\wd0}\kern -0.8\wd0}}%
          \mathopen{}\int}
\def\vintslides{\mathop{\mathchoice%
          {\setbox0\hbox{$\displaystyle\intop$}\kern 0.22\wd0%
           \vcenter{\hrule height 0.04em width 0.6\wd0}\kern -0.82\wd0}%
          {\setbox0\hbox{$\textstyle\intop$}\kern 0.2\wd0%
           \vcenter{\hrule height 0.04em width 0.6\wd0}\kern -0.8\wd0}%
          {\setbox0\hbox{$\scriptstyle\intop$}\kern 0.2\wd0%
           \vcenter{\hrule height 0.04em width 0.6\wd0}\kern -0.8\wd0}%
          {\setbox0\hbox{$\scriptscriptstyle\intop$}\kern 0.2\wd0%
           \vcenter{\hrule height 0.04em width 0.6\wd0}\kern -0.8\wd0}}%
          \mathopen{}\int}
\newcommand{\Bplus}{B_\limplus}
\newcommand{\Bminus}{B_\limminus}

\DeclareMathOperator{\Div}{div}
\DeclareMathOperator{\capp}{cap}

\newcommand{\cpp}{\capp_{p,|t|^a}}
\newcommand{\cpa}{\capp_{|t|^a}}
\newcommand{\capB}{\capp_{\Bsp}}
\newcommand{\capBtwo}{\capp_{\Bstwo}}
\newcommand{\Bsp}{B^s_p}
\newcommand{\Bstwo}{B^s_2}
\newcommand{\uP}{\itoverline{P}} 
\newcommand{\lP}{\itunderline{P}} 

\DeclareMathOperator{\spt}{supp}
\DeclareMathOperator{\pv}{p.v.}

\DeclareMathOperator*{\osc}{osc}
\newcommand{\bdry}{\partial}
\newcommand{\grad}{\nabla}
\newcommand{\bdy}{\bdry}
\newcommand{\simge}{\gtrsim}
\newcommand{\simle}{\lesssim}

{\catcode`p =12 \catcode`t =12 \gdef\eeaa#1pt{#1}}      
\def\accentadjtext#1{\setbox0\hbox{$#1$}\kern   
                \expandafter\eeaa\the\fontdimen1\textfont1 \ht0 }
\def\accentadjscript#1{\setbox0\hbox{$#1$}\kern 
                \expandafter\eeaa\the\fontdimen1\scriptfont1 \ht0 }
\def\accentadjscriptscript#1{\setbox0\hbox{$#1$}\kern   
                \expandafter\eeaa\the\fontdimen1\scriptscriptfont1 \ht0 }
\def\accentadjtextback#1{\setbox0\hbox{$#1$}\kern       
                -\expandafter\eeaa\the\fontdimen1\textfont1 \ht0 }
\def\accentadjscriptback#1{\setbox0\hbox{$#1$}\kern     
                -\expandafter\eeaa\the\fontdimen1\scriptfont1 \ht0 }
\def\accentadjscriptscriptback#1{\setbox0\hbox{$#1$}\kern 
                -\expandafter\eeaa\the\fontdimen1\scriptscriptfont1 \ht0 }
\def\itoverline#1{{\mathsurround0pt\mathchoice
        {\rlap{$\accentadjtext{\displaystyle #1}
                \accentadjtext{\vrule height1.593pt}
                \overline{\phantom{\displaystyle #1}
                \accentadjtextback{\displaystyle #1}}$}{#1}}
        {\rlap{$\accentadjtext{\textstyle #1}
                \accentadjtext{\vrule height1.593pt}
                \overline{\phantom{\textstyle #1}
                \accentadjtextback{\textstyle #1}}$}{#1}}
        {\rlap{$\accentadjscript{\scriptstyle #1}
                \accentadjscript{\vrule height1.593pt}
                \overline{\phantom{\scriptstyle #1}
                \accentadjscriptback{\scriptstyle #1}}$}{#1}}
        {\rlap{$\accentadjscriptscript{\scriptscriptstyle #1}
                \accentadjscriptscript{\vrule height1.593pt}
                \overline{\phantom{\scriptscriptstyle #1}
                \accentadjscriptscriptback{\scriptscriptstyle #1}}$}{#1}}}}
\def\itunderline#1{{\mathsurround0pt\mathchoice
        {\rlap{$\underline{\phantom{\displaystyle #1}
                \accentadjtextback{\displaystyle #1}}$}{#1}}
        {\rlap{$\underline{\phantom{\textstyle #1}
                \accentadjtextback{\textstyle #1}}$}{#1}}
        {\rlap{$\underline{\phantom{\scriptstyle #1}
                \accentadjscriptback{\scriptstyle #1}}$}{#1}}
        {\rlap{$\underline{\phantom{\scriptscriptstyle #1}
                \accentadjscriptscriptback{\scriptscriptstyle #1}}$}{#1}}}}
\newcommand{\Ga}{\Gamma}
\newcommand{\eps}{\varepsilon}
\newcommand{\Om}{\Omega}
\newcommand{\al}{\alpha}
\newcommand{\p}{{$p\mspace{1mu}$}}   
\newcommand{\R}{\mathbf{R}}
\newcommand{\vt}{\tilde{v}}
\newcommand{\ft}{\tilde{f}}
\newcommand{\fb}{\fh}
\newcommand{\fh}{\hat{f}}

\newcommand{\limplus}{{\mathchoice{\vcenter{\hbox{$\scriptstyle +$}}}
  {\vcenter{\hbox{$\scriptstyle +$}}}
  {\vcenter{\hbox{$\scriptscriptstyle +$}}}
  {\vcenter{\hbox{$\scriptscriptstyle +$}}}
}}
\newcommand{\limminus}{{\mathchoice{\vcenter{\hbox{$\scriptstyle -$}}}
  {\vcenter{\hbox{$\scriptstyle -$}}}
  {\vcenter{\hbox{$\scriptscriptstyle -$}}}
  {\vcenter{\hbox{$\scriptscriptstyle -$}}}
}}

\newcommand{\cprime}{$'$}

\makeatletter
\newcommand{\setcurrentlabel}[1]{\def\@currentlabel{#1}}
\makeatother

\numberwithin{equation}{section}

\newenvironment{ack}{\medskip{\it Acknowledgement.}}{}

\begin{document}

\authortitle{Jana Bj\"orn}
{Boundary estimates and a Wiener criterion \\
for the fractional Laplacian}
  \author{Jana Bj\"orn \\
\it\small Department of Mathematics, Link\"oping University, SE-581 83 Link\"oping, Sweden\\
\it \small jana.bjorn@liu.se, ORCID\/\textup{:} 0000-0002-1238-6751
}

\date{}
\maketitle

\begin{center}
 Dedicated to Vladimir Maz\cprime ya on his 85th birthday.
\end{center}

\bigskip

\noindent{\small {\bf Abstract}
Using the Caffarelli--Silvestre extension, we show for a general open set
$\Om\subset\R^n$ that a boundary point $x_0$ is regular for the  
fractional Laplace equation $(-\Delta)^su=0$, $0<s<1$, 
if and only if $(x_0,0)$ is regular for the extended weighted equation
in a subset of $\R^{n+1}$.
As a consequence, we 
characterize regular boundary points 
for $(-\Delta)^su=0$ by a Wiener 
criterion involving a Besov capacity.  
A decay estimate for the solutions near 
regular boundary points and the Kellogg property are also obtained.
} 



\medskip
\noindent
{\small Conflicts of interest: None.}

\section{Introduction}

We assume throughout the paper that $0<s<1$ and that $\Om\subset\R^n$,
$n\ge2$,  is an 
open set such that its complement $\R^n\setm\Om$ has positive Besov
$\Bstwo$-capacity, as given in~\eqref{eq-Bes-cap-intro} and Definition~\ref{def-capB}.
Note that we do not require any additional assumptions about the
regularity of $\Om$, which is allowed to be unbounded.

Consider the Dirichlet boundary value problem in
$\Om$ for the fractional equation 
\begin{equation}   \label{eq-frac-Lap-intro}
(-\Delta)^su=0.
\end{equation}
Recall that
up to a multiplicative constant, the fractional Laplacian is given by
the principle value integral
\[
(-\Delta)^su(x) := C_{n,s} \pv \int_{\R^n} \frac{u(x)-u(y)}{|x-y|^{n+2s}}\,dy.
\]
Solutions of   
$(-\Delta)^su=0$ coincide with the so-called $\al$-harmonic functions
(with $\al=2s$), defined by means of balayage 
and associated with the Riesz potentials $|x|^{\al-n}$, 
as in Bliedtner--Hansen~\cite[Chapter~V.4]{BliHan}, Hoh--Jacob~\cite{HohJac} and
Landkof~\cite[Chapter~IV.5]{Landkof}.

The fractional Laplacian is a nonlocal operator and hence
the Dirichlet boundary data for~\eqref{eq-frac-Lap-intro}
are prescribed on the complement $\R^n\setm\Om$, 
rather than on the boundary~$\bdy\Om$.
The above assumption that the complement  has positive capacity is natural 
since otherwise the complement and the boundary data are not seen
by the Besov space $\Bstwo(\R^n)$, associated with the fractional Laplacian.

We study the boundary behaviour and regularity of the solutions 
of~\eqref{eq-frac-Lap-intro},
i.e.\  whether every solution of the Dirichlet problem
for~\eqref{eq-frac-Lap-intro} in $\Om$, with 
continuous boundary data $f$,  attains its boundary value as the limit 
\[
\lim_{\Om\ni x\to x_0} u(x) =f(x_0) \quad \text{at } x_0\in\bdy\Om.
\]

The following sufficient and necessary condition for 
regular boundary points is proved in Section~\ref{sect-Dir-bdy-reg}.

\begin{thm}  \label{thm-frac-Wiener-Besov}
{\rm (Wiener criterion)}  
A boundary point $x_0\in\bdry\Om$ is regular for\/ $\Om$ with
respect to the fractional equation~\eqref{eq-frac-Lap-intro} if and
only if
\begin{equation}    \label{eq-Wien-cond-intro}
\int_0^1 \frac{\capBtwo(F_r,B(x_0,2r))}{r^{n-2s}}
         \,\frac{dr}{r} = \infty,
\end{equation}
where $F_r=\itoverline{B(x_0,r)}\setm \Om$ and $\capBtwo$ is the 
condensor capacity defined by a Besov seminorm on $\R^n$ as 
\begin{equation}   \label{eq-Bes-cap-intro}  
\capBtwo(F_r,B(x_0,2r)) = \inf_v \int_{\R^n} \int_{\R^n}
   \frac{|v(x)-v(y)|^2}{|x-y|^{n+2s}}\,dx\,dy,
\end{equation}
where the infimum is taken over all $v\in C^\infty_0(B(x_0,2r))$
such that $v\ge1$ on $F_r$.
\end{thm}

For regular boundary points $x_0\in\bdy\Om$, we obtain the pointwise
decay estimate
\begin{equation}   \label{eq-osc-intro}
\osc_{B(x_0,\rho)}u \le \osc_{F_{2R}}f + \Bigl( \osc_{\R^n\setm\Om}f \Bigr)
    \exp\biggl( - C\int_\rho^R \frac{\capBtwo(F_r,B(x_0,2r))}{r^{n-2s}}
         \,\frac{dr}{r} \biggr) 
\end{equation}
for some constant $C>0$ and all $0<\rho\le R<\infty$, see
Proposition~\ref{prop-decay-est} for a more precise formulation.
It follows that $u$ is H\"older continuous at $x_0$ whenever
 $\R^n\setm\Om$ has a \emph{corkscrew}
at $x_0$, in the sense that there are $0<c<1$ and $r_0>0$ such that 
for all $0<r\le r_0$, the set $F_r$
contains a ball of radius $cr$, cf.\ Lemma~\ref{lem-cpa-Brho}.

Decay estimates of the type~\eqref{eq-osc-intro} 
first appeared for solutions of $\Delta u=0$
in Maz\cprime ya~\cite{Maz-conf} (and for nonlinear equations in Maz$'$ya~\cite{Maz70}), 
where they were used to obtain the sufficiency part of the Wiener
criterion for such equations, as well as  
sufficient conditions for the H\"older continuity of the 
solutions at the boundary.

Regular points for the $\al$-fine potential theory ($\al=2s$) associated with Riesz
potentials were characterized by the Wiener criterion
in Bliedtner--Hansen~\cite[Corollary~V.4.17]{BliHan} 
and Landkof~\cite[Theorem~5.2]{Landkof}.
Proposition~VII.3.1 in \cite{BliHan} connects such potential-theoretic regular points
to regular boundary points for $\al$-harmonic functions in the sense
of our Definition~\ref{def-reg-frac} and~\cite[Chapter~VII.3]{BliHan}.

We instead use the extension results from
Caffarelli--Silvestre~\cite{CafSil} to derive the Wiener criterion 
\eqref{eq-Wien-cond-intro}
for the nonlocal equation~\eqref{eq-frac-Lap-intro} from the Wiener criterion
for local degenerate (weighted) divergence equations in $\R^{n+1}$.
This is the content of the following theorem, proved in Section~\ref{sect-Dir-bdy-reg}.
As a consequence, we also obtain the \emph{Kellogg property}, 
saying that the set of irregular
boundary points for equation~\eqref{eq-frac-Lap-intro} has zero
$\Bstwo$-capacity, see Corollary~\ref{cor-Kellogg}.

\begin{thm}   \label{thm-equiv-reg-intro}
Assume that $\R^n\setm\Om$ has positive $\Bstwo$-capacity. 
A boundary point $x_0\in\bdry\Om$ is regular for $\Om$ with
respect to the fractional equation~\eqref{eq-frac-Lap-intro} if and
only if $z_0:=(x_0,0)$ is regular for $G$ with respect to the weighted
equation~\eqref{eq-def-degen-eq}.
\end{thm}

Our condition~\eqref{eq-Wien-cond-intro} differs somewhat from the ones
in~\cite{BliHan} and~\cite{Landkof} in that we use the variational  (condensor)
capacity~\eqref{eq-Bes-cap-intro} and that our sets $F_r$ are defined
using balls rather than annuli, cf.\ \cite[(5.1.7)]{Landkof}.
In addition to the  Caffarelli--Silvestre extension, our results are
based on the following comparison between 
the Besov capacity and a weighted condensor capacity in $\R^{n+1}$,
see Lemma~\ref{lem-comp-cpa-capB} and Section~\ref{sect-cap}.

\begin{lem}   \label{lem-equiv-cap-intro}
Let $E\subset \itoverline{B(x_0,r)} \subset\R^n$ be a Borel set 
and $z_0=(x_0,0)\in\R^{n+1}$. 
Then for $p>1$  and $0<s<1$, 
\begin{equation*}   
\capB(E,B(x_0,2r))  \simeq  \cpp\bigl(E\times\{0\},B(z_0,2r)\bigr),
\quad \text{where } a=p(1-s)-1,
\end{equation*}
and $\cpp$ is the condensor capacity   associated with the weight
$w(x,t)=|t|^{a}$ in $\R^n\times \R$.
\end{lem}

Condition~\eqref{eq-Wien-cond-intro} is a fractional analogue of the
famous Wiener criterion proved for the Laplace equation
(i.e.\ \eqref{eq-frac-Lap-intro} with $s=1$) by Wiener~\cite{Wiener-crit} in 1924.
The Wiener criterion has been extended to various linear and nonlinear elliptic equations by
e.g.\ Littman--Stampacchia--Weinberger~\cite{LiStWe},
Maz$'$ya~\cite{Maz70}, Gariepy--Ziemer~\cite{GaZi},
Dal Maso--Mosco~\cite{DalMasoMosco86}, \cite{DalMasoMosco87}, 
Lindqvist--Martio~\cite{LinMar} and Kil\-pe\-l\"ai\-nen--Mal\'y~\cite{KilMaly}.
A weighted version of the Wiener criterion for degenerate
elliptic equations, which will be of great importance in this paper,
was proved by Fabes--Jerison--Kenig~\cite{FaJeKe},
Heinonen--Kilpel\"ainen--Martio~\cite{HeKiMa} and Mikkonen~\cite{Mikkonen}.

Similar sufficient conditions were obtained for the boundary
continuity of solutions with zero boundary data for the nonhomogeneous
polyharmonic equation
\[
(-\Delta)^mu=f\in C^\infty_0(\Om)
\]
with some integer powers
$m\ge2$, see Maz$'$ya~\cite{MazBiharm}, \cite{MazPoly} and
Maz$'$ya--Donchev~\cite{MazDonchev}.
Condition~\eqref{eq-Wien-cond-intro} (with a slightly different capacity) was
shown to guarantee boundary continuity of solutions with zero boundary
data for the nonhomogeneous fractional Laplace equation 
\[
(-\Delta)^su=f\in C^\infty_0(\Om) \quad \text{with } 
s\in(0,1)\cup(\tfrac{n}{2}-1,\tfrac{n}{2}),
\] 
see Eilertsen~\cite{Eilert}.
Estimates similar to~\eqref{eq-osc-intro} were also proved in the above papers 
\cite{Eilert},  \cite{MazBiharm}, \cite{MazPoly} and~\cite{MazDonchev}
on fractional and higher order equations,  but
necessity does not seem to have been considered there.
For an extensive exposition of results on boundary regularity and the
Wiener criterion for a wide class of elliptic equations, see the monograph by
Maz$'$ya~\cite{MazEMSBook}.

In the literature, the Dirichlet problem for $(-\Delta)^s$
is often considered in smooth (or at least Lipschitz) domains with
zero boundary data and for the nonhomogeneous equation
$(-\Delta)^su=f$. 
Formally and for suffiently smooth data, this formulation and the one
considered in this paper (with zero 
right-hand side and general boundary data) can be transformed into
each other, see e.g.\ Hoh--Jacob~\cite[Section~5]{HohJac} and Ros-Oton~\cite[Section~7]{RosPubMat}.

The boundary Harnack inequality for the fractional Laplacian was proved by Bogdan~\cite{Bogdan} and
Bogdan--Kulczycki--Kwa\'{s}nicki~\cite{BogKulKwa}.
Optimal regularity up to the boundary was in sufficiently smooth
bounded domains (Lipschitz or $C^{1,1}$) proved by
Ros-Oton--Serra~\cite{RosSer}, \cite{RosSerARMA} for solutions of
$(-\Delta)^su=f$ with zero boundary data  and various right-hand sides~$f$.

In this paper, we treat general open sets (with complements of
positive capacity).
Our approach to boundary regularity is based on the following 
result due to Caffarelli--Silvestre~\cite[Section~4]{CafSil}:
The solution
$u$ of \eqref{eq-frac-Lap-intro} in $\Om$ with $u=f$ on 
$F:=\R^n\setm\Om$ 
coincides with the restriction 
\[
u(x):=U(x,0)
\] 
of the solution $U$ to the Dirichlet problem in 
\begin{equation}   \label{def-G}
G:=\R^{n+1}\setm(F\times\{0\})
\end{equation}
 for the weighted equation
\begin{equation}   \label{eq-def-degen-eq}
\Div(|t|^{1-2s}\grad U(x,t))=0   
\end{equation}
with boundary data $f$ on $\bdy G=(F\times\{0\})\cup\{\infty\}$.

The above relation between~\eqref{eq-frac-Lap-intro}
and~\eqref{eq-def-degen-eq} was used in~\cite{CafSil} to derive the
Harnack and boundary Harnack inequalities for~\eqref{eq-frac-Lap-intro}. 
In particular, the local H\"older continuity of $U$ in $G$, proved in
e.g.\ Fabes--Kenig--Serapioni~\cite{FKS} 
or Heinonen--Kilpel\"ainen--Martio~\cite[Theorem~6.6]{HeKiMa},
directly yields interior regularity for the solutions of the fractional
equation~\eqref{eq-frac-Lap-intro}.
Since then, the lift from~\eqref{eq-frac-Lap-intro} to~\eqref{eq-def-degen-eq}  
has become a standard tool in the analysis of the fractional Laplacian.
It has been successfully exploited by many authors in various
contexts, including interior regularity and free boundaries, see e.g.\ 
Aimar--Beltritti--G\'omez~\cite{AimBelGom},
Barrios--Figalli--Ros-Oton~\cite{BarFigRos},
Caffarelli--Roquejoffre--Sire~\cite{CafRoqSire},
Caffarelli--Salsa--Silvestre~\cite{CafSalSilFreeObst},
Koch--R\"uland--Shi~\cite{KoRuShi} and  Silvestre~\cite{Silv}.

\begin{ack}
The author was supported by the Swedish Research Council,
grant 2018-04106.
\end{ack}

\section{Weights, capacities and degenerate equations}

\label{sect-cap}

In this section we discuss properties of the degenerate
equation~\eqref{eq-def-degen-eq} and its associated weighted capacity.
More generally, we let $p>1$ and consider the weight $w(x,t)=|t|^a$ in $\R^{n+1}$
with $-1<a<p-1$. 
Even though we only need the case $p=2$ in the rest of this paper, we
state and prove the general results in this section for all $p>1$, since the
arguments are the same as for $p=2$ and may be of independent interest. 

It can be verified by a direct calculation that $w$ is a Muckenhoupt $A_p$ weight on
$\R^{n+1}$, i.e.\ it satisfies for all balls $B\subset\R^{n+1}$,
\[
\int_B w(x,t)\,dx\,dt \biggl( \int_B w(x,t)^{1/(1-p)}\,dx\,dt \biggr)^{p-1}
    \le C |B|^p,
\]
where $|B|$ stands for the $(n+1)$-dimensional Lebesgue measure of $B$
and $C>0$ is independent of $B$.
It follows that $w$ is admissible for the theory of degenerate
elliptic equations in the sense of
Fabes--Jerison--Kenig~\cite{FaJeKe}  ($p=2$) or 
Heinonen--Kilpel\"ainen--Martio~\cite[Chapter~20]{HeKiMa} ($p>1$).

In~\cite{FaJeKe}, the arguments are restricted to a large fixed ball
$\Sigma =\{z:|z|<R\}$. 
Since we later deal with the unbounded domain $G$, given by~\eqref{def-G},
it will be convenient to use~\cite{HeKiMa} even for $p=2$.
Equation~\eqref{eq-def-degen-eq} satisfies the assumptions
(3.3)--(3.7) in~\cite{HeKiMa} with $p=2$ and the tools therein are
therefore at our disposal.

Here and in what follows, we use the notation $z=(x,t)\in\R^n\times\R=\R^{n+1}$
and define the measure $\mu_a$ on $\R^{n+1}$ by
\[
d\mu_a(z) = |t|^a\,dx\,dt.
\]
It follows from \cite[p.~307]{HeKiMa} that $\mu_a$ supports the
$(p,p)$-Poincar\'e inequality
\begin{equation}   \label{eq-(p,p)-PI}
\int_B |v-v_B|^p\,d\mu_a \le C r^p \int_B |\grad v|^p\,d\mu_a,
\end{equation}
whenever $B=B(z,r)\subset\R^{n+1}$ 
is a ball and $v\in C^\infty(B)$ is bounded.
Here 
\[
v_B:= \vint_B v\,d\mu_a := \frac{1}{\mu_a(B)} \int_B v\,d\mu_a
\]
is the integral average of $v$ and $C>0$ is independent of $B$.
Since we consider balls both in $\R^{n+1}$ and $\R^n$, we adopt the convention 
that the dimension of a ball is determined by its centre, i.e.\ 
for $z\in\R^{n+1}$ and $x\in\R^{n}$,
\[
B(z,r):=\{y\in\R^{n+1}:|y-z|<r\} \quad \text{and} \quad
 B(x,r):=\{y\in\R^{n}:|y-x|<r\}.
\]
Unless specified otherwise, we consider open balls.

The following definition of capacity follows \cite[Chapter~2]{HeKiMa}.

\begin{deff}  \label{def-cpa}
Let $B\subset \R^{n+1}$ be a ball and $K\subset B$ be a compact set.
The \emph{weighted variational {\rm(}condensor\/{\rm)} capacity} of $K$ with respect to $B$ is
\[
\cpp(K,B) = \inf_v \int_{B} |\grad v|^p \, \,d\mu_a,
\]
where the infimum is taken over all $v\in C^\infty_0(B)$ such that
$v\ge1$ on $K$.
\end{deff}

The weighted capacity $\cpp$ extends to all subsets of $B$ as a Choquet
capacity and in particular, for all Borel sets $E\subset B$,
\[
\cpp(E,B) = \sup_{\text{compact }K\subset E} \cpp(K,B).
\]
A set $E\subset\R^{n+1}$ is said to be of zero $\cpp$-capacity if for all
balls $B\subset \R^{n+1}$,
\[
\cpp(E\cap B,B)=0.
\] 
In \cite{HeKiMa}, the variational capacity $\cpp$ is used to
characterize boundary regularity for weighted equations of \p-Laplace
type and in particular for the equation~\eqref{eq-def-degen-eq}.

Our aim is to formulate the Wiener criterion in terms of
a capacity associated with the Besov space $\Bsp(\R^n)$ and the
fractional equation~\eqref{eq-frac-Lap-intro}.
Following Maz$'$ya~\cite[p.~512]{MazSobBook}, we define the Besov seminorm
\[
\|v\|_{\Bsp(\R^n)} := 
   \biggl( \int_{\R^n} \int_{\R^n}
   \frac{|v(x)-v(y)|^p}{|x-y|^{n+sp}}\,dx\,dy \biggr)^{1/p}, \quad 0<s<1<p.
\]
Theorem~1 in \cite[p.~512]{MazSobBook} 
asserts that for all $v\in C^\infty_0(\R^n)$,
\begin{equation}   \label{eq-comp-Maz-seminorms}
\|v\|_{\Bsp(\R^n)} 
\simeq \inf_{\vt} \bigl\| |t|^{1-s-1/p} \grad \vt \bigr\|_{L^p(\R^{n+1})},
\end{equation}
where the infimum is taken over all extensions $\vt\in
C^\infty_0(\R^{n+1})$ of $v$.
Moreover, it follows from the proof of~\cite[Theorem~1 on p.~512]{MazSobBook} that
\begin{equation}   \label{eq-comp-Maz-extension}
\|v\|_{\Bsp(\R^n)} 
\simeq \bigl\| |t|^{1-s-1/p} \grad V \bigr\|_{L^p(\R^{n+1})},
\end{equation}
where the extension $V$ is defined by
\begin{equation}   \label{eq-def-V}
V(x,t) = \frac{1}{t^n} \int_{\R^n} \phi \Bigl( \frac{\xi-x}{t} \Bigr) v(\xi)\,d\xi,
\end{equation}
with any $0\le\phi\in C^\infty(\R^n)$ vanishing outside $B(0,1)$, such that
\[
\int_{\R^n} \phi(\xi)\,d\xi=1 \quad \text{and} \quad \phi(-\xi)=\phi(\xi),
\]
see~\cite[(10.1.7) on p.~514]{MazSobBook}.
Here, and in what follows, we use the notation $X\simeq Y$ if there is
a positive constant $C$ independent of $X$ and $Y$ such that $X/C\le Y \le CX$.
Similar one-sided inequalities are denoted $\simle$ and $\simge$ in an
obvious way.

The following lemma relates the weighted capacity $\cpp$ in $\R^{n+1}$
to a Besov type condensor capacity in $\R^n$, see Definition~\ref{def-capB}.
For simpler notation, we identify $K\subset\R^n$ with $K\times\{0\}\subset\R^{n+1}$.

\begin{lem}  \label{lem-comp-cpa-capB}
Let $z_0=(x_0,0)\in\R^{n+1}$ and 
$K\subset \itoverline{B(x_0,r)}\subset\R^n$ be a compact set.
Then
\[
\cpp(K,B(z_0,2r)) \simeq \inf_v \|v\|^p_{\Bsp(\R^n)},
\quad \text{where } s=1-\frac{1}{p}-\frac{a}{p},
\]
and the infimum is taken over all $v\in C^\infty_0(B(x_0,2r))$
such that $v\ge1$ on~$K$.
\end{lem}

\begin{proof}
Note that $a/p = 1-s-1/p$ coincides with the exponent in
\eqref{eq-comp-Maz-seminorms} and \eqref{eq-comp-Maz-extension}.
Let $\vt\in C^\infty_0(B(z_0,2r))$ be admissible in the definition of $\cpp(K,B(z_0,2r))$
and set $v(x):=\vt(x,0)$ for $x\in\R^n$.
It then follows from~\eqref{eq-comp-Maz-seminorms} and the definition
of $\mu_a$ that
\[
\int_{B(z_0,2r)} |\grad \vt|^p \, d\mu_a \simge \|v\|_{\Bsp(\R^n)}.
\]
Since $v\in C^\infty_0(B(x_0,2r))$  and $v\ge1$ on $K$,
taking infimum over all such functions $\vt$ gives the $\simge$-inequality 
in the statement of the lemma.

For the reverse inequality, let $v\in C^\infty_0(B(x_0,2r))$ 
be such that $v\ge1$ on~$K$.
Since $\phi$ in~\eqref{eq-def-V} vanishes outside $B(0,1)$, it is easily verified that
the extension $V$, given by~\eqref{eq-def-V}, satisfies 
\begin{equation}  \label{eq-supp-V}
V(x,t)=0 \quad \text{whenever}  \quad  |x-x_0|\ge 2r+t.
\end{equation}
To estimate $\cpp(K,B(z_0,2r))$, let
$\eta\in C^\infty_0(B(z_0,2r))$ be  a cut-off function such that $0\le\eta\le1$ in
$\R^{n+1}$, $\eta=1$ in $B(z_0,r)$ and $|\grad\eta|\le 2/r$.
We then have
\begin{align}
\cpp(K,B(z_0,2r)) &\le \int_{B(z_0,2r)} |\grad (V\eta)|^p \, d\mu_a
\label{eq-est-with-V-eta}  \\ 
&\le 2^p \int_{B(z_0,2r)} |\grad V|^p \, d\mu_a  
    + \frac{4^p}{r^p} \int_{B(z_0,2r)} |V|^p \, d\mu_a. 
\nonumber
\end{align}
In view of~\eqref{eq-comp-Maz-extension}, it therefore suffices to
estimate the last term using the first integral on the right-hand side.
To this end, we let $B=B(z_0,3r)$ and 
$V_B:=\vint_{B}V \, d\mu_a$.  
The Minkowski inequality then yields
\begin{equation}   \label{eq-est-with-V-B}
   \biggl( \vint_{B} |V|^p \, d\mu_a \biggr)^{1/p} 
   \le \biggl( \vint_{B} |V-V_B|^p \, d\mu_a \biggr)^{1/p} + |V_B|. 
\end{equation}
Note that, by the H\"older inequality and~\eqref{eq-supp-V},
\begin{align*}
|V_B| \le  \vint_{B}|V|\,d\mu_a
   &\le \biggl( \vint_{B}|V|^p \,d\mu_a \biggr)^{1/p} 
      \biggl( \frac{\mu_a(B \cap\spt V)}{\mu_a(B)} \biggr)^{1-1/p} \\
   &\le \theta \biggl( \vint_{B}|V|^p \,d\mu_a \biggr)^{1/p},
\end{align*}
where the constant $0<\theta<1$ depends only on $n$, $p$ and $a$.
Inserting the last estimate into~\eqref{eq-est-with-V-B}
and subtracting the last term from the left-hand side, we get
\[
(1-\theta) \biggl( \int_{B} |V|^p \, d\mu_a \biggr)^{1/p} 
   \le \biggl( \int_{B} |V-V_B|^p \, d\mu_a \biggr)^{1/p}.
\]
Together with~\eqref{eq-est-with-V-eta} and the $(p,p)$-Poincar\'e
inequality~\eqref{eq-(p,p)-PI} for $\mu_a$, this implies that
\begin{align*}
\cpp(K,B(z_0,2r)) &\le 2^p \int_{B(z_0,2r)} |\grad V|^p \, d\mu_a  
    + \frac{4^p C}{(1-\theta)^p} \int_{B} |\grad V|^p \, d\mu_a   \\
&\simle \int_{\R^{n+1}} |\grad V|^p \, d\mu_a.
\end{align*}
The comparison~\eqref{eq-comp-Maz-extension} now concludes the proof.
\end{proof}

In view of Lemma~\ref{lem-comp-cpa-capB} and Definition~\ref{def-cpa}, 
we make the following definition.

\begin{deff}  \label{def-capB}
Let $B\subset \R^n$ be a ball and $K\subset B$ be a compact set.
The \emph{variational {\rm(}condensor\/{\rm)} Besov capacity} of $K$ with respect to $B$ is
\[
\capB(K,B) = \inf_v \|v\|^p_{\Bsp(\R^n)},
\]
where the infimum is taken over all $v\in C^\infty_0(B)$ such that $v\ge1$ on $K$.
For a Borel set $E\subset B$, we let
\[
\capB(E,B) = \sup_{\text{compact }K\subset E} \capB(K,B).
\]
We also say that a Borel set $E\subset\R^n$ has zero $\Bsp$-capacity if 
\[
\capB(E\cap B,B)=0 \quad \text{for all balls $B\subset \R^n$.}
\] 
\end{deff}

Lemma~\ref{lem-comp-cpa-capB} now implies that whenever 
$E\subset \itoverline{B(x_0,r)} \subset\R^n$ is a Borel set and $z_0=(x_0,0)\in\R^{n+1}$,
\begin{equation}    \label{eq-comp-cpa-capB}
\cpp(E,B(z_0,2r)) \simeq \capB(E,B(x_0,2r)),
\quad \text{where } s=1-\frac{1}{p}-\frac{a}{p}.
\end{equation}
This proves Lemma~\ref{lem-equiv-cap-intro}.
In particular, \eqref{eq-comp-cpa-capB} holds for the sets 
$F_r$ in Theorem~\ref{thm-frac-Wiener-Besov}.
Moreover, a Borel subset of $\R^n$ has zero $\Bsp$-capacity if and
only if it has zero $\cpp$-capacity.

\begin{lem}  \label{lem-cpa-Brho}
Let $0<\rho\le r$, $z_0=(x_0,0)$ and $s=1-1/p-a/p$. Then 
\begin{align*}    
\cpp(B(z_0,\tfrac12\rho),B(z_0,2r)) &\simle \capB(\itoverline{B(x_0,\rho)},B(x_0,2r)) \\
&\simle \cpp(B(z_0,\rho),B(z_0,2r)).
\end{align*}
In particular, $\capB(\itoverline{B(x_0,cr)},B(x_0,2r)) \simeq r^{n-ps}$ with
comparison constants depending only on $n$, $p$, $s$ and\/ $0<c\le1$.
\end{lem}

\begin{proof}
Let $v\in C^\infty_0(B(x_0,2r))$ be such that $v\ge1$ on
$\itoverline{B(x_0,\rho)}$.
Then
\[
\|v\|^p_{\Bsp(\R^n)}\simeq \int_{\R^{n+1}} |\grad V|^p\,d\mu_a,
\]
where the extension $V$ is given by~\eqref{eq-def-V}.
Note that as in~\eqref{eq-supp-V}, we have
\[
\text{$V(x,t)=0$ if $|x-x_0|\ge 2r+t$}
\quad \text{and} \quad 
\text{$V(x,t)=1$ if $|x-x_0|\le \rho-t$.}
\]
In particular, $V(x,t)=1$ on $B(z_0,\tfrac12\rho)$ and hence, as in
the proof of Lemma~\ref{lem-comp-cpa-capB},  
\[
\cpp(B(z_0,\tfrac12\rho),B(z_0,2r)) \simle \int_{\R^{n+1}} |\grad V|^p\,d\mu_a.
\]
Taking infimum over all functions $v$,  which are 
admissible in the definition of $\capB(B(x_0,\rho),B(x_0,2r))$,
proves the first inequality in the statement of the lemma.
The second inequality follows from \eqref{eq-comp-cpa-capB} and the inclusion
$B(x_0,\rho)\times\{0\}\subset B(z_0,\rho)$, together with the equality
\[
\cpp(\itoverline{B(z_0,\rho)},B(z_0,2r)) = \cpp(B(z_0,\rho),B(z_0,2r)),
\]
cf.~\cite[p.~32]{HeKiMa}.

As for the last statement, it can be proved in the same
way as in
\cite[Lemma~2.14]{HeKiMa} that the weighted condensor capacity of balls in
$\R^{n+1}$ satisfies
\begin{equation}  \label{eq-est-cap-B}
\cpp(B(z_0,cr),B(z_0,2r)) \simeq \frac{1}{r^p} \int_{B(z_0,r)} |t|^a \,dx\,dt 
    \simeq r^{n+1+a-p} = r^{n-ps},
\end{equation}
where the comparison constants in $\simeq$ depend on $n$, $p$, $s$ and $c$,
but are independent of $z_0$ and $r$.
The first part of the lemma  with $\rho=cr$ then concludes the proof.
\end{proof}

\section{The Dirichlet problem and boundary regularity}
\label{sect-Dir-bdy-reg}

In this section, we let $p=2$ and $a=1-2s$, where $0<s<1$.
Note that $s=\tfrac12(1-a)$ and so Lemmas~\ref{lem-comp-cpa-capB} and \ref{lem-cpa-Brho},
as well as the comparisons~\eqref{eq-comp-cpa-capB} and~\eqref{eq-est-cap-B}, 
apply with this choice of parameters.
We write $\cpa$ instead of $\capp_{2,|t|^a}$.

The Dirichlet problem for $(-\Delta)^su=0$
was solved  by the Perron method on general open sets $\Om\subset\R^n$ and for continuous
boundary data vanishing at infinity by 
Bliedtner--Hansen~\cite[Chapter~VII]{BliHan} and Hoh--Jacob~\cite[Section~4]{HohJac}.
Solutions obtained as minimizers of energy integrals with sufficiently regular boundary 
data appear e.g.\ in
Felsinger--Kassmann--Voigt~\cite{FelKaVoi}.
Solutions defined using balayage 
were considered by  Landkof~\cite[Chapter~IV.5]{Landkof}.

As mentioned in the introduction, it follows from
Caffarelli--Silvestre~\cite[Section~4]{CafSil} that the Dirichlet problem for 
the fractional equation  $(-\Delta)^su=0$
in $\Om\subset\R^n$ can be seen as a
restriction of the Dirichlet problem for the \emph{weighted
equation}~\eqref{eq-def-degen-eq} in 
\[
G=\R^{n+1}\setm(F\times\{0\}), \quad 
\text{where} \quad F=\R^n\setm\Om.
\]
More precisely, assume that 
$f$ is continuous on $F$ and vanishes outside some bounded set.
As before, for simpler notation, we identify $F$ with $F\times\{0\}=\R^{n+1}\setm G=\bdy G$.
The solution $U$ of the Dirichlet problem in $G$ for the weighted
equation 
\begin{equation}   \label{eq-def-deg-a}
\Div(|t|^a\grad U(x,t))=0   
\end{equation}
with boundary data $f$ on $\bdy G$ then exists and can be obtained by the Perron 
method, as in Heinonen--Kilpel\"ainen--Martio~\cite[Definition~9.1]{HeKiMa}.
We point out that the point at $\infty$ in the
Dirichlet problem for~\eqref{eq-def-deg-a} on $G$
is considered as a part of the boundary of $G$ and
that~$f$ is continuous also at $\infty$, with $f(\infty)=0$.
Recall that we assume that $F$ has positive $\Bstwo$-capacity. 
The estimate~\eqref{eq-comp-cpa-capB} then implies that also $\cpa(F\times\{0\})>0$.

The \emph{upper Perron solution} in $G$ for an arbitrary bounded function
$f$ defined on $\bdy G\cup\{\infty\}$ is 
\begin{equation}  \label{eq-def-uP}
\uP_G f(z) := \inf_v v(z), \quad z\in G,
\end{equation}
with the infimum taken over all lower semicontinuously regularized supersolutions
$v$ of~\eqref{eq-def-deg-a} in $G$, which are bounded from below and
satisfy
\[
\liminf_{G\ni y\rightarrow z} v(y)   \ge f(z)\quad\text{for all }
z\in\bdy G \cup\{\infty\},
\]
see \cite[Theorem~7.25 and Definition~9.1]{HeKiMa}.

The \emph{lower Perron solution} is defined similarly using
upper semicontinuously regularized subsolutions
of~\eqref{eq-def-deg-a} or by $\lP_G f:= -\uP_G(-f)$. 
It follows directly from the definition of Perron solutions
that if $f_1\le f_2$ on $\bdy G$ then the corresponding Perron
solutions satisfy $\uP_G f_1 \le \uP_G f_2$ and
$\lP_G f_1 \le \lP_G f_2$ in $G$.
Moreover, the comparison principle between regularized 
sub- and supersolutions \cite[p.~133]{HeKiMa} yields that 
$\lP_G f\le \uP_G f$ for every $f$.

Since $\cpa(\R^{n+1}\setm G)>0$, every continuous function $f$ on
$\bdy G\cup\{\infty\}$ is \emph{resolutive}, i.e.\ 
$\uP_G f = \lP_G f$, see \cite[Theorem~9.25]{HeKiMa}.
The Perron solution will therefore be denoted $P_G f$. 
By~\cite{CafSil}, the restriction 
\(
u(x):= P_G f(x,0)
\)
satisfies the fractional equation 
\begin{equation}   \label{eq-def-frac-Lap}
(-\Delta)^su=0
\end{equation}
 in $\Om$.
Moreover, by the Kellogg property \cite[Theorem~9.11]{HeKiMa} 
and~\eqref{eq-comp-cpa-capB},
\[
\lim_{y\to x} u(y) = \lim_{z\to (x,0)} P_G f(z) = f(x)
\]
holds for all $x\in F$ outside a set of zero $\Bstwo$-capacity.
This function $u$ is a bounded solution of the Dirichlet
problem for $(-\Delta)^su=0$ with boundary data $f$ on $F$.

\begin{deff}   \label{def-reg-frac}
A boundary point $x_0\in\bdry\Om$ is \emph{regular} for $\Om$ with
respect to the fractional equation~\eqref{eq-def-frac-Lap} if for
each $f\in C(\R^n\setm\Om)$  vanishing outside some bounded set,
the Perron solution $u$ of~\eqref{eq-def-frac-Lap} 
with boundary data $f$ on $\R^n\setm\Om$ satisfies
\begin{equation}   \label{eq-def-ref-u}
\lim_{\Om\ni x\to x_0} u(x) =f(x_0).
\end{equation}
A point is \emph{irregular} if it is not regular.
\end{deff}

Following~\cite[p.~171]{HeKiMa}, we say that a point $z_0\in\bdy G$ is
\emph{regular} for $G$ with respect to the equation~\eqref{eq-def-deg-a} if 
for each boundary data $f\in C(\bdy G\cup\{\infty\})$,
the Perron solution $P_G f$ satisfies
\begin{equation}   \label{eq-reg-HKM}
\lim_{G\ni z\to z_0} P_G f(z) = f(z_0).
\end{equation}
We are now ready to prove the equivalence between the above two notions
of regular boundary points.

\begin{proof} [Proof of Theorem~\ref{thm-equiv-reg-intro}]
Assume  that $z_0$ is regular for $G$ and let $f$ be as in
Definition~\ref{def-reg-frac}. 
Then clearly, by the definition of $u$ and by~\eqref{eq-reg-HKM},
\[
\lim_{\Om\ni x\to x_0} u(x) = \lim_{G\ni z\to z_0} P_G f(z) =f(x_0).
\]
Hence, $x_0$ is regular for $\Om$.

Conversely, assume that $z_0$ is not regular for $G$. 
Then there exists a continuous function $f$ on $\bdry G\cup\{\infty\}$ such that the
corresponding Perron solution $P_G f$ 
of the weighted equation~\eqref{eq-def-deg-a} in $G$ with
boundary data $f$ fails~\eqref{eq-reg-HKM}.
By adding a constant and changing the sign, if needed, 
we can without
loss of generality assume that $f\ge0$ on $\bdy G$ and that
\begin{equation}   \label{eq-assume-u<f}
0\le \liminf_{G\ni z\to z_0} P_Gf(z) < f(z_0).
\end{equation}
Multiplying $f$ by a continuous cut-off function 
with compact support,  
we can moreover assume that $f$ vanishes outside some bounded set.
To conclude the proof, it suffices to show 
that also~\eqref{eq-def-ref-u} fails for the boundary data $f$.
So, assume for a contradiction that~\eqref{eq-def-ref-u} holds and define 
\[
\ft(z) = \begin{cases}  f(z), &\text{if } z\notin G,\\
                        P_Gf(z), &\text{if } z\in G.
           \end{cases}
\]
Let $B$ be a ball centred at $z_0$ and consider the upper half-ball 
\[
\Bplus:= \{(x,t)\in B: t>0\}.
\]
Then $\ft$ is a bounded function on $\bdy\Bplus$, which is continuous at $z_0$,
because of the assumption~\eqref{eq-def-ref-u}.

We claim that the restriction to $\Bplus$ of $P_Gf$ is the Perron
solution for $\ft$ in $\Bplus$.
Indeed, if $v$ is admissible in the definition~\eqref{eq-def-uP} of
$\uP_Gf=P_Gf$, then by the definition of $\ft$ and the continuity of $P_Gf$ in $G$, 
\[
\liminf_{\Bplus\ni y\to z}v(y) 
      \ge \begin{cases}  f(z)=\ft(z), &\text{if } z\in \bdy\Bplus\cap \bdy G,\\
                        {\displaystyle \lim_{\Bplus\ni y\to z} P_G(y) = \ft(z)}, &\text{if } z\in \bdy\Bplus\cap G.
           \end{cases}
\]
Hence
\[
v\ge \uP_{\Bplus} \ft  \quad \text{in } \Bplus,
\]
and taking infimum over all such $v$ shows that $P_Gf\ge \uP_{\Bplus}\ft$ in $\Bplus$.
Similarly, $P_G f\le \lP_{\Bplus}\ft$ in $\Bplus$.
Since also $\lP_{\Bplus}\ft \le \uP_{\Bplus}\ft$, we see that the
function $\ft$ is resolutive and $P_Gf$ is the Perron
solution of~\eqref{eq-def-deg-a} in $\Bplus$ with boundary data~$\ft$.

Now, by the corkscrew condition \cite[Theorem~6.31]{HeKiMa} (with
respect to $\R^{n+1}$), 
$z_0$ is a regular boundary point for $\Bplus$ and~\eqref{eq-def-deg-a}.
Since $\ft$ is bounded on $\bdy\Bplus$ and continuous at $z_0$, 
Lemma~9.6 in~\cite{HeKiMa} implies that 
\[
\lim_{\Bplus \ni z\to z_0} P_Gf(z) 
= \ft(z_0) = f(x_0).
\]
A similar argument applied to $\Bminus:= \{(x,t)\in B: t<0\}$,
together with the assumption that~\eqref{eq-def-ref-u} holds, then gives
\begin{equation*}   
\lim_{G\ni z\to z_0} P_Gf(z) = f(z_0),
\end{equation*}
which contradicts~\eqref{eq-assume-u<f} and concludes the proof.
\end{proof}

We can now make use of the Wiener criterion for the weighted equation
\eqref{eq-def-deg-a} in $G\subset\R^{n+1}$, provided by
Heinonen--Kilpel\"ainen--Martio~\cite[Theorem~21.30]{HeKiMa} or
Fabes--Jerison--Kenig~\cite{FaJeKe}.

\begin{proof}  
[Proof of Theorem~\ref{thm-frac-Wiener-Besov}]
By Theorem~\ref{thm-equiv-reg-intro}, the regularity of $x_0$ with respect 
to the fractional equation $(-\Delta)^su=0$ is equivalent
to the regularity of $z_0$ with respect to the weighted equation~\eqref{eq-def-deg-a}.
This is in turn, by the Wiener criterion~\cite[Theorem~21.30 and
  (6.17)]{HeKiMa} with $p=2$, equivalent to the condition
\[
\int_0^1 \frac{\cpa(F_r,B(z_0,2r))}{\cpa(B(z_0,r),B(z_0,2r))}
         \,\frac{dr}{r} = \infty.
\]
Here we have used the monotonicity of $\cpa$, together with estimates 
similar to the proof of \cite[Lemma~2.16]{HeKiMa}, to replace
$\bdy G \cap B(z_0,r)$ by the compact set 
$F_r=\itoverline{B(x_0,r)}\setm\Om$ in the above integral.
Finally, the estimates~\eqref{eq-comp-cpa-capB}   and~\eqref{eq-est-cap-B} 
(with $p=2$) conclude the proof.
\end{proof}

\medskip

We conclude the paper with two additional properties of regular boundary
points for the fractional equation $(-\Delta)^s u=0$.

\begin{cor}[Kellogg property]  \label{cor-Kellogg}
The set of irregular boundary points for the fractional equation $(-\Delta)^s u=0$ has
$\Bstwo$-capacity zero.
\end{cor}

\begin{proof}
This follows immediately from Theorem~\ref{thm-equiv-reg-intro}, together 
with~\eqref{eq-comp-cpa-capB} and the 
Kellogg property \cite[Theorem~9.11]{HeKiMa} for the weighted equation~\eqref{eq-def-deg-a}.
\end{proof}

\begin{prop}  \label{prop-decay-est}
Assume that $x_0\in\bdy\Om$ is regular for $(-\Delta)^s u=0$ and that
$f\in C(\R^n\setm\Om)$ vanishes outside some bounded set.
Let $u$ be the Perron solution  of $(-\Delta)^s u=0$
in $\Om$ with boundary data $f$ on $\R^n\setm\Om$. 
Then for all $0<\rho\le R<\infty$,
\[
\sup_{\Om\cap B(x_0,\rho)}u \le \sup_{F_{2R}}f + \Bigl( \sup_{\R^n\setm\Om}f -
\sup_{F_{2R}}f \Bigr)
    \exp\biggl( - C\int_\rho^R \frac{\capBtwo(F_r,B(x_0,2r))}{r^{n-2s}}
         \,\frac{dr}{r} \biggr), 
\]
where $F_r=\itoverline{B(x_0,r)}\setm \Om$ and $C$ depends only on $n$ and $s$.
\end{prop}

\begin{proof}
We shall use \cite[Theorem~6.18]{HeKiMa}, where a similar decay
estimate is proved for the weighted equation~\eqref{eq-def-deg-a} in
bounded domains and with continuous Sobolev boundary data. 
Recall that $u$ is the restriction to $\Om$ of the Perron solution $P_Gf$
for~\eqref{eq-def-deg-a} in $G=\R^{n+1}\setm((\R^n\setm\Om)\times\{0\})$.
Since $G$ is unbounded and the Perron solution in general only belongs to a
local Sobolev space, we proceed as follows.
 
Let $\eps>0$, $m_\eps=\sup_{F_{2R+\eps}}f$ and $M=\sup_{\R^n\setm\Om}f$.
The Perron solution $P_Gf$ of~\eqref{eq-def-deg-a}
clearly satisfies $P_Gf\le M$.
Find $\fb_\eps \in C^\infty(\R^{n+1})$ such that
\[
\text{$\fb_\eps=m_\eps=\min_{\R^{n+1}}\fb_\eps$ in $B(z_0,2R)$}
\quad \text{and} \quad
\text{$\fb_\eps=M=\max_{\R^{n+1}}\fb_\eps$ outside $B(z_0,2R+\eps)$.}
\]
Then $\fb_\eps\ge f$ on $\R^{n+1}\setm G$.
Theorem~6.18 in \cite{HeKiMa}, applied to $\fb_\eps$ and the bounded set 
$G\cap B(z_0,2R+\eps)$, gives as in the proof of Theorem~\ref{thm-frac-Wiener-Besov} that
\begin{align*}
&\sup_{\Om\cap B(x_0,\rho)} u \le \sup_{G\cap B(z_0,\rho)} P_Gf \le
\sup_{G\cap B(z_0,\rho)} P_G \fb_\eps \\
&\le \fb_\eps(z_0) + \osc_{F_{2R}} \fb_\eps 
+ (M-m_\eps)
    \exp\biggl( - C\int_\rho^R \frac{\cpa(F_r,B(z_0,2r))}{\cpa(B(z_0,r),B(z_0,2r))}
         \,\frac{dr}{r} \biggr). 
\end{align*}
Since $\osc_{F_{2R}} \fb_\eps =0$ and 
$\fb_\eps(z_0)=m_\eps\to \sup_{F_{2R}}f$ as $\eps\to0$,
the estimates~\eqref{eq-comp-cpa-capB} and~\eqref{eq-est-cap-B}
(with $p=2$) conclude the proof.
\end{proof}

\end{document}